\documentclass[ijoc,sglanonrev]{informs4}

\RequirePackage{tgtermes}
\RequirePackage{newtxtext}
\RequirePackage{newtxmath}
\RequirePackage{bm}
\RequirePackage{endnotes}
\usepackage{enumitem}
\OneAndAHalfSpacedXII

\usepackage{algorithm}
\usepackage{algpseudocode}
\usepackage{tikz}

\usepackage{natbib}
 \bibpunct[, ]{[}{]}{,}{n}{}{,}%
 \def\bibfont{\small}%

\DeclareMathOperator{\dist}{D}

\newcommand{\f}[1]{F(x^{#1})}
\newcommand{\g}[1]{\nabla F(x^{#1})}

\EquationsNumberedThrough
\TheoremsNumberedThrough
\ECRepeatTheorems

\MANUSCRIPTNO{IJOC-0001-2024.00}

\makeatletter
\renewcommand{\theARTICLETOP}{}   % no "Submitted to ..." block
\makeatother

\usepackage{fancyhdr}
\newcommand{\ShortAuthors}{A. Peivasti and Zamani}
\newcommand{\ShortTitle}{Convergence rate of the DRS algorithm}
\fancypagestyle{plain}{%
  \fancyhf{}
  \fancyfoot[C]{\small \thepage}

}

\begin{document}
%%%%%%%%%%%%%%%%

%\setcounter{page}{1}
%\DOI{10.1287/xxxx.0000.0000}%

\RUNAUTHOR{A. Peivasti and Zamani}

% \RUNTITLE{Predictive Maintenance in Manufacturing}
\RUNTITLE{Convergence rate of the DRS algorithm}

% \TITLE{Optimal Resource Allocation in Humanitarian Logistics: A Stochastic Programming Approach}
\TITLE{On the convergence rate of the Douglas-Rachford splitting algorithm}

\ARTICLEAUTHORS{%
%\AUTHOR{John Doe,\textsuperscript{a} Jane Smith,\textsuperscript{b}}
%\AFF{\textsuperscript{a}Department of Industrial Engineering, University of XYZ, \EMAIL{john.doe@xyz.edu; \textsuperscript{b}Department of Computer Science, University of ABC, \EMAIL{jane.smith@abc.edu}} 

\AUTHOR{Hadi Abbaszadehpeivasti}
\AFF{Department of Econometrics and Operations Research, Tilburg University, Tilburg, The Netherlands, \EMAIL{h.peivasti@tilburguniversity.edu}}

\AUTHOR{Moslem Zamani}
\AFF{Department of Industrial Engineering and Innovation Sciences, Eindhoven University of Technology, Eindhoven, The Netherlands, \EMAIL{m.zamani@tue.nl}}

} % end of the block

\ABSTRACT{%
% Enter your abstract
This work is concerned with the convergence rate analysis of the Dou- glas–Rachford splitting (DRS) method for finding a zero of the sum of two maximally monotone operators. We obtain an exact rate of convergence for the DRS algorithm and demonstrate its sharpness in the setting of convex feasibility problems. Further- more, we investigate the linear convergence of the DRS algorithm, providing both necessary and sufficient conditions that characterize this behavior. We further examine the performance of the DRS method when applied to convex composite optimization problems. The paper concludes with several conjectures on the convergence behavior of the DRS algorithm for this class of problems.
}%

%\FUNDING{This research was supported by [grant number, funding agency].}

%Supplemental Material:
%Data Ethics & Reproducibility Note:

%\KEYWORDS{Stochastic programming, Decision support,Uncertainty, Disaster response, Optimization}

\KEYWORDS{Douglas-Rachford splitting method,  Convergence rate, Performance estimation} 
%\HISTORY{Received: Month DD, YYYY; Accepted: Month DD, YYYY; Published Online: Month DD, YYYY}

\maketitle
\thispagestyle{plain}
%%%%%%%%%%%%%%%%%%%%%%%%%%%%%%%%%%%%%%%%%%%%%%%%%%%%%%%%%%%%%%%%%%%%%%

% Text of your paper here

\section{Introduction}
\label{sec:intro}
We consider the following monotone inclusion problem 
\begin{align}\label{P}
\underset{x\in\mathbb{R}^n}{\text{find}}\quad 0\in Ax+Bx,
\end{align}
where $A, B : \mathbb{R}^n \rightrightarrows \mathbb{R}^n$ are maximally monotone operators.
The monotone inclusion problem  provides a unifying framework that captures numerous fundamental problems in convex optimization, variational inequalities,  and equilibrium models \cite{bauschke2017, combettes2018monotone,ryu2022large}.

Among the iterative schemes devised for solving monotone inclusions, the Douglas–Rachford splitting (DRS) algorithm \cite{combettes2018monotone} has emerged as one of the most effective and versatile methods. Its popularity stems from its ability to decompose complex problems into simpler subproblems.
Furthermore, several important algorithms, including the alternating direction method of multipliers (ADMM), can be analyzed through the lens of the DRS algorithm; see \cite[Chapter 3]{ryu2022large} for further details. Indeed, the close connection of the DRS method to other fundamental operator-splitting techniques, notably ADMM, further highlights its central role in modern optimization \cite{ryu2022large}.

The DRS method is originally introduced for solving linear systems that appeared in the numerical treatment of partial differential equations  (heat equation) \cite{douglas1956numerical}.  
Lions and Mercier \cite{lions1979splitting} established that the method is also convergent for a general maximal monotone inclusion problem in the form of \eqref{P}. Eckstein and Bertsekas \cite{eckstein1992douglas} showed that the DRS method can be interpreted as a proximal point method.  We refer the interested reader to \cite{glowinski2017splitting} for a historical review of the DRS. The method is given in Algorithm \ref{DRS}; see Section \ref{sec:preliminaries} for the notations used in Algorithm \ref{DRS}. 

\begin{algorithm}
\caption{The Douglas-Rachford splitting method}
\begin{algorithmic}
\State \textbf{Parameters:} number of iterations $N$,  positive stepsize $\gamma>0$ and relaxation factor $\lambda\in (0, 2)$.
\State \textbf{Inputs:} maximally monotone operators $A$ and $B$, initial iterate $w^1 \in \mathbb{R}^n$.
\State For $k=1, 2, \ldots, N$ perform the following steps:\\
\begin{enumerate}[label=\roman*)]
\item 
$x^k = J_{\gamma B}(w^k)$.
\item
$y^k = J_{\gamma A}(2x^k - w^k)$.
\item
$w^{k+1} = w^k + \lambda (y^k - x^k)$.
\end{enumerate}
\end{algorithmic}
\label{DRS}
\end{algorithm}

Algorithm \ref{DRS} can be written as a fixed point iteration $w^{k+1}=Tw^k$, where 
\begin{align}\label{DR_ope}
T=(1-\tfrac{\lambda}{2})I+\tfrac{\lambda}{2}R_{\gamma A}R_{\gamma B}.  
\end{align}
and $R$ denotes the reflected resolvent transformation. We call $T$  the Douglas-Rachford (DR) operator throughout the text. It is worth noting that $x^\star$ is a solution of problem \eqref{P}  if and only if there exists $w^\star$ such that $x^\star=J_{\gamma B}w^\star$ and $w^\star$ is a fixed point of the DR operator. We denote the fixed points of operator $T$ by $W^\star$.   

 We adopt the following assumptions to establish convergence rate results.
 
\begin{assumption}\label{assumption}
Operators $A$ and $B$ are maximally monotone operators and 
the DR operator \eqref{DR_ope} for the given parameters admits a fixed point $w^*$.
\end{assumption}

Note that under these assumptions, $w^k\to w^\star$ for some $w^\star\in W^\star$; see \cite[Theorem 26.11]{bauschke2017}. In this manuscript, we employ the performance estimation framework to establish our main results. Performance estimation provides a strong and unifying approach for the analysis and design of first-order methods. The framework is introduced in the seminal work of Drori and Teboulle \cite{drori2014performance}, and has since attracted significant attention in the optimization community. For a comprehensive review of the performance estimation, we refer the interested reader to \cite{abbaszadehpeivasti2024performance, taylor2024towards, taylor2017smooth}.

This paper is structured as follows. Section \ref{sec:preliminaries} introduces the definitions and fundamental concepts that will be used throughout the subsequent sections. In Section \ref{Sub_s}, we analyze the sublinear convergence of the DRS method. Section \ref{Lin_S} is dedicated to the study of its linear convergence. Finally, in Section \ref{Comp_S}, we investigate the behavior of the DRS method in the context of convex composite optimization problems, and we conclude the paper with some conjectures on the convergence rate
of the DRS algorithm for this class of problems.

\section{Terminology and notation} \label{sec:preliminaries}
We briefly recall the key definitions employed in this paper. For further details, the reader is referred to standard references  \cite{bauschke2017,ryu2022large}. 

We denote the Euclidean inner product and  norm by $\langle \cdot , \cdot \rangle$ and $\|\cdot\|$, respectively. We use $I$ to denote the identity operator. 
For a (set-valued) operator $A : \mathbb{R}^n \rightrightarrows \mathbb{R}^n$, we define  its graph by $\operatorname{gra} A=\{(x,u): u\in Ax\}$ and we denote its inverse by $A^{-1}$.
An operator $A : \mathbb{R}^n \rightrightarrows \mathbb{R}^n$ is called monotone if 
\[
\langle u-v, x-y \rangle \geq 0 \quad \forall (x,u), (y,v) \in \operatorname{gra}A,
\]
and it is called maximally monotone if its graph is not properly contained in that of any other monotone operator.  The operator $A$ is $\beta$-cocoercive if for some $\beta>0$, we have
\[
\langle u-v, x-y \rangle \geq \beta\|u-v\|^2 \quad \forall (x,u), (y,v) \in \operatorname{gra}A.
\]
It is seen that cocoercivity implies Lipschitz continuity but its converse does not necessarily hold. Consequently,  a cocoercive operator is single-valued. The operator $A$ is called nonexpansive if it is Lipschitz continuous with modulus one.  An operator $A$ is called $\mu$-strongly monotone if its inverse is $\mu$-cocoercive. 
The resolvent of operator $A$ is denoted and defined as $J_{A}= (I + A)^{-1}$. In addition, we use $R_A=2J_A-I$ to denote the reflected resolvent of $A$.

  Let $f:\mathbb{R}^n\to(-\infty, \infty]$ be an extended convex function.
The function $f$ is called closed if its epigraph is closed,
i.e. $\{(x, r): f(x)\leq r\}$ is a closed subset of $\mathbb{R}^{n+1}$. The function $f$ is said to be proper if there exists $x\in\mathbb{R}^n$ with $f(x)<\infty$. The subgradients of $f$ at $x$ are denoted and defined as
$$
\partial f(x)=\{u: f(y)\geq f(x)+\langle u, y-x\rangle, \forall y\in\mathbb{R}^n\}.
$$
When $f$ is a closed proper convex function, $\partial f$ is  a  maximally monotone operator. 
The convex function $f$ is called $L$-smooth if $\partial f$ is $\tfrac{1}{L}$-cocoercive. The function $f$ is said to be $\mu$-strongly convex if $\partial f$ is $\mu$-strongly monotone. The proximal operator of a closed proper convex function $f$ is defined and denoted as $\operatorname{prox}_{f}=(I+\partial f)^{-1}$.

Given closed convex set $C$, the normal cone operator $N_C$ of $C$ is defined by 
\[
N_C(x) = \{u \in \mathbb{R}^n : \langle u, y-x \rangle \leq 0, \ \forall y \in C\}, \quad x \in C,
\]
with $N_C(x) = \emptyset$ otherwise. 
The projection onto  $C$ is written as $\Pi_C$. The corresponding distance function to $C$ is denoted by $\dist_C$.

%%%%%%%%%%%%%%%%%%%%%%%%%%%%%%%%%
\section{Sublinear convergence of the DRS method} \label{Sub_s}

In this section, we study the convergence of the DRS algorithm for a general monotone inclusion problem. Regarding the DRS algorithm, we have 
 \begin{align}\label{OO}
\left\|Tw^N - w^N\right\|^2 \leq  \tfrac{\lambda(N-1)^{(N-1)}}{(2-\lambda)N^N} \left\|w^1 - w^\star\right\|^2,
\end{align}
when $\gamma >0$ and $\lambda\in [1, 1+\sqrt{\tfrac{k-1}{k}})$ in Algorithm \ref{DRS}; see \cite[Theorem 4.9]{lieder2018projection}. Indeed, this result is established for the Krasnoselski-Mann iteration for a general nonexpansive operator, and the convergence rate of the DRS algorithm emerges as a specific instance. Since the proof in \cite{lieder2018projection} is rather technical and lengthy (spanning seven pages), we provide here a proof for the case $\lambda=1$  to keep the presentation self-contained. We present the theorem for a general nonexpansive operator and the convergence rate of the DRS follows from it as $R_{\gamma A}R_{\gamma B}$, the Peaceman–Rachford operator, is nonexpansive \cite{ryu2022large}. Note that the DR operator for $\lambda=1$ is 
$T=\tfrac{1}{2}I+\tfrac{1}{2}R_{\gamma A}R_{\gamma B}$.

\begin{theorem} \label{thm:main_rate}
Let $S:\mathbb{R}^n\to\mathbb{R}^n$ be a nonexpansive operator and let $S$ have a fixed point $w^\star$, i.e. $Sw^\star=w^\star$.  Suppose that $w^1\in\mathbb{R}^n$ the sequence $\{w^k\}$ is generated by $w^{k+1}=(\tfrac{1}{2}I+\tfrac{1}{2}S)w^k$. Then  
\begin{align}\label{T.C}
\left\|\tfrac{1}{2}Sw^N - \tfrac{1}{2}w^N\right\|^2 \leq  \tfrac{(N-1)^{(N-1)}}{N^N} \left\|w^1 - w^\star\right\|^2.
\end{align}
\end{theorem}
\begin{proof}{Proof.}
To establish the convergence rate, we demonstrate its validity by summing a series of valid inequalities. Without loss of generality, we assume that $w^\star=0$ (consider the nonexpansive operator  given by $S(w+w^\star)-w^\star$).  Due to nonexpansivity of $S$,  we have
\begin{align*}
\left\|w^{k+1}-w^{k} \right\|^2-\left\|2w^{k+2}-w^{k+1}-(2w^{k+1}-w^{k}) \right\|^2
\geq 0, \ \
\left\|w^k-w^\star \right\|^2-\left\|2w^{k+1}-w^k-w^\star \right\|^2
\geq 0,
\end{align*}
where we use $Sw^k=2w^{k+1}-w^k$ and $Sw^\star=w^\star$.
Summing these inequalities with the indicated multipliers yields
\begin{align*}
&\sum_{k=1}^{N-1} \tfrac{k(N-1)^{N-k-1}}{2N^{N-k}}\left(\left\|w^{k+1}-w^{k} \right\|^2-\left\|2w^{k+2}-w^{k+1}-(2w^{k+1}-w^{k}) \right\|^2 \right)
%%%%%%%%%%%%%%%%
+\tfrac{1}{2N}\left(  \left\|w^N \right\|^2-\left\|2w^{N+1}-w^N \right\|^2 \right)\\
%%%%%%%%
&+\sum_{k=1}^{N-2} \tfrac{(N-k-1)(N-1)^{N-k-1}}{2N^{N+1-k}}\left(  \left\|w^k \right\|^2-\left\|2w^{k+1}-w^k \right\|^2 \right)=
 %%%%%%%%
 \tfrac{(N-1)^{N-1}}{N^N}\left\| w^1-w^\star\right\|^2-\left\| w^{N+1}-w^N\right\|^2\\
 %%%%
& -\left\| w^{N+1}-\tfrac{2(N-1)}{N}w^N+\tfrac{N-1}{N}w^{N-1}\right\|^2
%%%%
-\sum_{k=1}^{N-2} \tfrac{k(N-1)^{N-2-k}}{N^{N-k-1}}\left\| w^{k+2}-\tfrac{2(N-1)}{N}w^{k+1}+\tfrac{N-1}{N}w^{k}\right\|^2
%%%%%%%%%%
\geq 0.
\end{align*}
The above inequality implies the convergence rate \eqref{T.C} as $w^{N+1}=\tfrac{1}{2}Sw^N+\tfrac{1}{2}w^N$ and the proof is complete. 
\Halmos
\end{proof}

In the following, we demonstrate that convergence rate \eqref{OO} is tight by constructing a convex feasibility problem that achieves this rate.

\begin{theorem} \label{thm_Lower}
Let $\gamma > 0$, $\lambda = 1$, and $N\geq 2$ in Algorithm \ref{DRS}. Then there exist maximally monotone operators $A$ and $B$ such that 
after $N$ iterations of the algorithm, we have 
\begin{align*}
\left\|T(w^N) - w^N\right\|^2 =\tfrac{(N-1)^{\,N-1}}{N^N} \,\|w^1 - w^\star\|^2,
\end{align*}
where $w^\star$ denotes a fixed point of the DR operator.
\end{theorem}

\begin{proof}{Proof.}
 Consider the following subspaces in $\mathbb{R}^2$,
 $$
 P=\{(x, 0): x\in \mathbb{R}\},  \ \ 
 Q=\{(x, \tfrac{x}{\sqrt{N-1}}): x\in \mathbb{R}\}.
 $$
Consider the following problem,
$$
\underset{x\in\mathbb{R}^2}{\text{find}}\quad 0\in N_P(x)+N_Q(x).
$$
Assume $A$ and $B$ denote the normal cone mappings of $Q$ and  $P$, respectively. For the given problem, the DR operator is
$$
T=\sqrt{\tfrac{N-1}{N}}\begin{pmatrix}
    \cos\theta & -\sin\theta\\
    \sin\theta &  \cos\theta
\end{pmatrix},
$$
where $\theta=\arcsin\left(\tfrac{1}{\sqrt{N}}\right)$. It is easily seen that zero is the solution to the problem and the fixed point of  operator $T$. Consider  $w^1=\begin{pmatrix}
    \cos\phi & \ \sin\phi
\end{pmatrix}^T$, for some arbitrary $\phi$.
It is seen
\[
\left\|w^{N}\right\| = \left(\frac{N-1}{N}\right)^{\tfrac{N-1}{2}}, \qquad
\left\|w^{N+1}\right\|  = \left(\frac{N-1}{N}\right)^{\tfrac{N}{2}},
\]
 and the angle between them is $\theta$. By the law of cosines, we get
 \begin{align}\label{eq_b}
 \left\|w^{N+1}-w^{N}\right\|^2=\tfrac{(N-1)^{\,N-1}}{N^N},
 \end{align}
 and the proof is complete.
 \Halmos
\end{proof}

It is worth noting that $\theta$ used in the above theorem is equivalent to the Friedrichs angle between $P$ and $Q$. A notable implication of Theorem~\ref{thm_Lower} is that the convergence rate of the DRS algorithm for convex feasibility problem cannot exceed that of a general monotone inclusion problem. Moreover, in the proof of Theorem~\ref{thm_Lower}, if we select $\phi$ such that $w^N\in P$, we have 
$$
\dist_P(y^N)=\sqrt{\tfrac{(N-1)^{\,N-1}}{N^N}}.
$$
This bound  resembles  the rate established in \cite[Theorem 4.4]{zamani2024exact} for the alternating projection method. 

One may wonder if convergence rate \eqref{OO} can be improved in the presence of cocoercivity. The answer is negative. Assume that $N\geq 2$ is given. Consider the following problem,
$$
\underset{x\in\mathbb{R}^2}{\text{find}}\quad 0 \in Ax+Bx,
$$
where $A, B:\mathbb{R}^2\to \mathbb{R}^2$, $B$ is the zero operator and 
$Ax=\begin{pmatrix}
    \tfrac{-x_2}{\sqrt{N-1}} &  \tfrac{x_1}{\sqrt{N-1}}
\end{pmatrix}^T$. 
It is readily seen that both $A$ and $B$ are monotone and $B$ is cocoercive. In addition, the DR operator is the same as that given in Theorem~\ref{thm_Lower} for $\lambda=1$. Therefore, it follows from the theorem we have  \eqref{eq_b} for any unit vector $w^1\in\mathbb{R}^2$.

As the aforementioned example illustrates, cocoercivity of $B$ is not a sufficient condition to improve convergence rate \eqref{OO}. However, it appears that if $A$ is a subdifferential operator, the rate can be improved. We conclude this section with a conjecture regarding this topic, which is informed by numerical experiments conducted within the performance estimation framework. 

\begin{conjecture} \label{Conjecture_1}
Let Assumption \ref{assumption} hold.  Assume that $B$ is a $\beta$-cocoercive operator and $A=\partial g$ for some closed proper convex function $g$. If the sequence $\{w^k\}$ is generated by Algorithm \ref{DRS} with stepsize $\gamma\in (0, \beta)$ and relaxation parameter $\lambda\in (0, 2)$, then 
\begin{align*}
\left\|T(w^N) - w^N\right\|^2 \leq  \tfrac{\lambda^2}{((N-1)\lambda+1)^2} \left\|w^1 - w^\star\right\|^2. 
\end{align*}
\end{conjecture}
%%%%%%%%%%%%%%%%%%%%%%%%%%%%%%%%%%%%%%
\section{Linear convergence  of the DRS method}\label{Lin_S}
In this section, we study the linear convergence of the DRS algorithm. 
The linear convergence rate of the DRS algorithm under strong monotonicity has been extensively studied in recent years, and exact rates have been established. Interested readers may refer to \cite{giselsson2017tight, giselsson2016linear,ryu2020operator} for further details.

As strong convexity is not commonly satisfied in many applications, alternative assumptions have been considered in the study of the linear convergence rate of first-order methods; see e.g. \cite{aspelmeier2016local,karimi2016linear,zamani2024convergence}. We study the linear convergence of the DRS algorithm under the error bound condition introduced in \cite{pena2021linear}. 

  \begin{definition}
Let $W\subseteq{R}^n$ and let $T$ denote the DR operator for some $\gamma>0$ and $\lambda\in (0, 2)$. The operator $T$ satisfies the error bound condition on $W$ if there exists $\mu\geq0$ such that 
\begin{align}\label{ER}
  \dist_{W^\star}(w)\leq \mu\left\| (I-T)w\right\|, \ \ \   \forall w\in W. 
\end{align}
  \end{definition}
  
It is worth noting that this definition resembles the error bound condition introduced in \cite{luo1992linear}. The following theorem investigates the convergence rate of the DRS algorithm in terms of the distance of $w^k$ to the fixed point set of the DR operator. 

\begin{theorem}\label{Th.L}
Let Assumption \ref{assumption} hold and let $\gamma>0$ and $\lambda\in (0, 2)$. Suppose that $\{(x^k, y^k, w^k)\}$ is generated by  Algorithm \ref{DRS} with initial point $w^1$.  If $T$ satisfies  error bound condition \eqref{ER} on $W=\{w: \dist_{W^\star}(w)\leq  \dist_{W^\star}(w^1)\}$, then 
\begin{align*}
\dist_{W^\star}(w^{k+1})\leq \sqrt{ 1-\tfrac{1}{\mu^2}(\tfrac{2}{\lambda}-1)}   \dist_{W^\star}(w^k).
\end{align*}
\end{theorem}
\begin{proof}{Proof.}
 Let $w^\star$ denote the nearest point to $w^k$ in $W^\star$, that is $\dist_{W^\star}(w^k)=\|w^k-w^\star\|$. By \cite[Theorem 1]{pena2021linear}, we have $w^k\in W$. Assume that $x^\star=J_{\gamma B}(w^\star)$. Therefore, we have $\tfrac{1}{\gamma}(w^\star-x^\star)\in Bx^\star$ and $\tfrac{1}{\gamma}(x^\star-w^\star)\in Ax^\star$. In addition,   
 $\tfrac{1}{\gamma}(w^k-x^k)\in Bx^k$ and $\tfrac{1}{\gamma}(2x^k-y^k-w^k)\in Ay^k$. By the monotonicity,  we have 
 \begin{align*}
&  \langle w^k-x^k - (w^\star-x^\star),  x^\star-x^k\rangle \leq 0,\\
&  \langle 2x^k-y^k-w^k - (x^\star-w^\star),  x^\star-y^k\rangle \leq 0.
 \end{align*}
By summing these inequalities, we get 
\begin{align}\label{ineq1}
\nonumber &\langle  w^k+y^k-x^k -w^\star, y^k-x^k\rangle=
\\
 & \tfrac{1}{\lambda}\langle  w^k-w^\star+\tfrac{1}{\lambda}(w^{k+1}-w^k), w^{k+1}-w^k\rangle\leq 0,
 \end{align}
where the last equality follows from $w^{k+1}=w^k+\lambda(y^k-x^k)$. By the error bound condition, we have 
\begin{align}\label{ineq2}
\left\|  w^k-w^\star \right\|^2 -\mu^2 \left\|  w^{k+1}-w^k \right\|^2 \leq 0.
\end{align}
By multiplying \eqref{ineq1} and \eqref{ineq2} by $2\lambda$ and $\tfrac{2-\lambda}{\lambda \mu^2}$, respectively,  we get
$$
\left\|  w^{k+1}-w^\star \right\|^2
\leq 
\left( 1-\tfrac{1}{\mu^2}(\tfrac{2}{\lambda}-1) \right) \left\|  w^k-w^\star \right\|^2.
$$
The above inequality implies the desired inequality since 
$ \|w^{k+1}-w^\star \|\geq \dist_{W^\star}(w^{k+1})$
and the proof is complete. 
\Halmos
\end{proof}

It is seen from the theorem that $\mu$ cannot take a value less than $\tfrac{2}{\lambda}-1$. Pe{\~n}a et al. established the linear convergence of the DRS algorithm when $\gamma=\lambda=1$; see \cite[Theorem 1]{pena2021linear}. In comparison with the above-mentioned result, Theorem \ref{Th.L} studies the convergence rate for all possible stepsizes and relaxation parameters. Moreover, it coincides  with their result for  $\gamma=\lambda=1$. 

In the following proposition, we establish that error bound condition \eqref{ER} is a necessary condition for the linear convergence in terms of $\dist_{W^\star}(w^k)$.  

\begin{proposition} \label{Th.LtoE}
Let Assumption \ref{assumption} hold and let $\gamma>0$ and $\lambda\in (0, 2)$. If Algorithm \ref{DRS} is linearly convergent with rate $r\in (0,1)$ in terms of $\dist_{W^\star}(w^k)$, then $T$ satisfies  error bound condition \eqref{ER} on $W=\{w: \dist_{W^\star}(w)\leq  \dist_{W^\star}(w^1)\}$ with 
$\mu=\frac{1}{1-r}$.
\end{proposition}
\begin{proof}{Proof.}
Let $w\in W\setminus W^\star$ and $\|Tw-w^\star\|=\dist_{W^\star}(Tw)$. As the algorithm is linearly convergent, we have
\[
\|Tw-w^\star\|^2\leq r^2\dist_{W^\star}^2(w) \;\le\; r^2\|w-w^\star\|^2.
\]
Therefore,
\begin{align*}
& (1-r^2)\|w-w^\star\|^2
\le \|w-w^\star\|^2-\|Tw-w^\star\|^2
\\
 &=  \bigl(\|w-w^\star\|+\|Tw-w^\star\|\bigr)\bigl(\|w-w^\star\|-\|Tw-w^\star\|\bigr)
 \\
 & \leq (1+r) \left\| w-w^\star \right\| \left\|  Tw-w \right\|,
\end{align*}
where the last inequality follows from the assumptions and the triangle inequality. 
By dividing both sides of the inequality by  $\|w-w^\star\|$, we get the desired inequality and complete the proof.
\Halmos
\end{proof}

Pe{\~n}a et al.\cite[Theorem 2]{pena2021linear} proved that  error bound condition \eqref{ER} holds under the following assumptions:
\begin{enumerate}[label=\roman*)]
 \item $B=\partial f$ and $A=\partial g$ for some proper convex functions $f$ and $g$;
\item  $f$ or $g$ is strongly convex;
\item  $f$ or $g$ is $L$-smooth.
\end{enumerate}

This result demonstrates that the error bound condition is less restrictive than the strong convexity assumption, which is commonly employed to establish linear convergence rate. In the following proposition, we show that the error bound condition may hold under assumptions more relaxed than those stated above. To this end, we employ restricted strong monotonicity notion.

\begin{definition}
Let $A$ and $B$ be maximal monotone operators. We say that the problem \eqref{P} satisfies restricted strong monotonicity if there exists $\mu_f> 0$ with
\begin{align}\label{Q_G_G_D}
\langle u+v, x-x^\star\rangle\ge \mu_f \|x-x^\star\|^2, \ \ \ 
\forall u\in Ax, \forall v\in Bx,
\end{align}
where $x^\star=\Pi_{X^\star}(x)$.
\end{definition}

\begin{proposition}
Let Assumption \ref{assumption} hold and let $B$ be $\beta$-cocoercive.  If   problem \eqref{P}  satisfies restricted strong monotonicity with modulus $\mu_f$, then $T$ satisfies  error bound condition \eqref{ER} on $W=\{w: \dist_{W^\star}(w)\leq  \dist_{W^\star}(w^1)\}$ with 
\[
\mu \;=\; \frac{\gamma + \gamma\min(\mu_f \beta, 1) + \beta}{\lambda\gamma\min(\mu_f \beta, 1)},
\]
for any $\lambda\in(0,2)$ and $0<\gamma\leq \beta$.
\end{proposition}
\begin{proof}{Proof.} % $0\mu_f\leq \beta$
 Without loss of generality, we assume \(\beta=1\) and $x^\star=0$, since problem~\eqref{P} is equivalent to $0\in \tfrac{1}{\beta}A(x-x^\star)+ \tfrac{1}{\beta}B(x-x^\star)$. Assume that $x^\star=J_{\gamma B}(w^\star)$. So $Bx^\star=\tfrac{1}{\gamma}(w^\star-x^\star)$ and $\tfrac{-1}{\gamma}(w^\star-x^\star)\in Ax^\star$. Suppose that $w^k\in W$ and 
\[
x^k=J_{\gamma B}w^k,\qquad y^k=J_{\gamma A}(2x^k-w^k),\qquad  w^{k+1}=T w^k=w^k+\lambda(y^k-x^k).
\]
We define $\bar \mu_f=\min(\mu_f,1)$, $u_y=By^k$ and
\begin{align}\label{XX}
u_x=\tfrac1\gamma(w^k-x^k)=Bx^k,\qquad
v_y=\tfrac1\gamma(2x^k-w^k-y^k)\in Ay^k.
\end{align}

Suppose that $S=1+\gamma(1+\bar\mu_f)$ and
\[
\alpha_1=\frac{2}{\bar\mu_f^2}\left(\gamma\left(1-\bar\mu_f^2\right)+\frac{1}{\gamma}+2\right),\qquad
\alpha_2=2\gamma S,\qquad
\alpha_3=\frac{2(\gamma+1)S}{\bar\mu_f}.
\]
Due to $\beta$-cocoercivity, we have 
$$
\langle u_x-u_y, x^k-y^k\rangle-\|u_x-u_y\|^2\geq 0, \ \ 
\langle u_y-Bx^\star, y^k-x^\star\rangle-\|u_y-Bx^\star\|^2\geq 0.
$$
By restricted strong monotonicity, we get
$$
\langle u_y+v_y,y^k-x^\star\rangle-\bar\mu_f\|y^k-x^\star\|^2\geq 0.
$$
Upon multiplying these inequalities by the specified multipliers and invoking \eqref{XX}, we derive
\begin{align*}
    &\alpha_3\big(\langle u_y+v_y,y^k-x^\star\rangle-\mu_f\|y^k-x^\star\|^2\big)
+\alpha_1\big(\langle u_x-u_y,\;x^k-y^k\rangle-\|u_x-u_y\|^2\big)\\
&
+\alpha_2\big(\langle u_y-Bx^\star,\;y^k-x^\star\rangle-\|u_y-Bx^\star\|^2\big)
=
\mu^2\| Tw^k-w^k\|^2-\|w^k-w^\star\|^2\\
& -D\left\|w-w^\star-\frac{4\gamma C}{D}u_x
-\frac{\gamma(\gamma+\bar\mu_{f}\gamma+1)+\tfrac{C}{\bar\mu_{f}}}{D}u_y
-\frac{2(\bar\mu_{f}\gamma+\tfrac{1}{2})C}{\bar\mu_{f}D}v_y\right\|^2\\
& -E\left\|u_x-\frac{P_{1}}{D'}u_y+
\frac{2\bar\mu_{f}\gamma^{3}(2\gamma+3)}{D'}v_y\right\|^2
-\frac{\gamma(\bar\mu_{f}-1)^{2}\,P_2}{\bar\mu_{f}^{2}D'}\left\|u_y-\frac{1}{\bar\mu_{f}-1}v_y\right\|^2,
\end{align*}
where
\begin{align*}
& C=(\gamma+1)\big((\bar\mu_{f}+1)\gamma+1\big), 
\qquad 
D  = 2C - 1,\\
& E=\frac{(\gamma+\bar\mu_{f}\gamma+1)^{2}}{\bar\mu_{f}^{2}}
+ 8\gamma^{2}C
- \frac{\Big(-2\gamma\bar\mu_{f}^{2}+\tfrac{2\gamma^{2}+4\gamma+2}{\gamma}\Big)(\gamma-1)}{\bar\mu_{f}^{2}}
- \frac{16\gamma^{2}C^{2}}{D},\\
& D'=2-2\bar\mu_{f}^{2}\gamma^{4} - 6\bar\mu_{f}^{2}\gamma^{3} - 4\bar\mu_{f}^{2}\gamma^{2} 
+ 4\bar\mu_{f}\gamma^{4} + 8\bar\mu_{f}\gamma^{3} + \bar\mu_{f}\gamma^{2} +2\bar\mu_{f}\gamma 
- 2\gamma^{4} - 2\gamma^{3} + 7\gamma^{2} + 9\gamma,\\
& P_{1}=2\bar\mu_{f}^{2}\gamma^{4} - 4\bar\mu_{f}^{2}\gamma^{2} + 2\bar\mu_{f}\gamma^{3} + \bar\mu_{f}\gamma^{2} + 2\bar\mu_{f}\gamma - 2\gamma^{4} - 2\gamma^{3} + 7\gamma^{2} + 9\gamma + 2,\\
& P_2=-3\bar\mu_{f}^{3}\gamma^{5} - 10\bar\mu_{f}^{3}\gamma^{4} - 8\bar\mu_{f}^{3}\gamma^{3} 
- 3\bar\mu_{f}^{2}\gamma^{5} - 7\bar\mu_{f}^{2}\gamma^{4} - 8\bar\mu_{f}^{2}\gamma^{3} - 4\bar\mu_{f}^{2}\gamma^{2}- \bar\mu_{f}\gamma^{5} + 4\bar\mu_{f}\gamma^{4} \\
\quad&  + 19\bar\mu_{f}\gamma^{3} + 22\bar\mu_{f}\gamma^{2} + 8\bar\mu_{f}\gamma- \gamma^{5} - 3\gamma^{4} + \gamma^{3} + 11\gamma^{2} + 12\gamma + 4.
\end{align*}
By doing some algebra, one can show that $D, E, P_2\geq 0$. Therefore,
\[
0\le\ \mu^2\,\|Tw^k-w^k\|^2-\|w^k-w^\star\|^2,
\]
where completes the proof.
\Halmos
\end{proof}

\section{Convex composite optimization problem} \label{Comp_S}

In this section, we study the DRS algorithm for the following convex composite optimization problem,
\begin{align}\label{P_O}
   \min\ f(x)+g(x),  
\end{align}
where $f, g:\mathbb{R}^n\to\mathbb{R}$ are closed proper convex functions. This problem may be regarded as a special case of problem~\eqref{P}, since it is equivalent to the monotone inclusion problem 
$$
\underset{x\in\mathbb{R}^n}{\text{find}}\quad 0 \in \partial f(x) + \partial g(x).
$$
Therefore, all the results established above remain valid for this setting. However, because problem~\eqref{P_O} is itself a particular instance of a monotone inclusion problem, one can derive sharper results in this setting.

The DRS algorithm for the convex composite optimization problem is given in Algorithm \ref{DRS_O}. 

\begin{algorithm}
\caption{The DRS algorithm for the composite optimization problem}
\begin{algorithmic}
\State \textbf{Parameters:} number of iterations $N$,  positive stepsize $\gamma>0$ and relaxation factor $\lambda\in (0, 2)$.
\State \textbf{Inputs:} closed proper convex functions $f$ and $g$, initial iterate $w^1 \in \mathbb{R}^n$.
\State For $k=1, 2, \ldots, N$ perform the following steps:\\
\begin{enumerate}[label=\roman*)]
\item 
$x^k = \operatorname{prox}_{\gamma f}(w^k)$.
\item
$y^k = \operatorname{prox}_{\gamma g}(2x^k - w^k)$.
\item
$w^{k+1} = w^k + \lambda (y^k - x^k)$.
\end{enumerate}
\end{algorithmic}
\label{DRS_O}
\end{algorithm}

 Throughout the section, we use the following assumptions. 
 
\begin{assumption}\label{assumption2}
We consider the following assumptions.
\begin{enumerate}[label=\roman*)]
    \item 
    $f$ and $g$ are closed proper convex functions.
    \item 
  The DR operator $T$ admits a fixed point \(w^\star\).
    \item 
    $f$ is $L$-smooth.
\end{enumerate}
\end{assumption}

Note that Assumption \ref{assumption2} implies that $x^\star=\operatorname{prox}_{\gamma f}(w^\star)$ is a solution to problem \eqref{P_O}. In the following conjecture, we give the exact convergence rate of the DRS algorithm. The rate is based on numerical experiments conducted within the performance estimation framework. \footnote{We hope that the conjectures presented in this paper will contribute to advancing research on the topic. Owing to external constraints, the authors have faced challenges in dedicating sufficient time to further investigation of this topic.}

\begin{conjecture} \label{Conjecture_2}
Let Assumption \ref{assumption2} hold. If the sequence $\{(x^k, y^k, w^k)\}$ is generated by the Algorithm \ref{DRS_O} with stepsize $\gamma\in (0, \tfrac{2\sqrt{2}-1}{L})$, relaxation parameter $\lambda\in (0, \tfrac{1+\sqrt{5}}{2})$ and the initial point $w^1$, then 
\begin{align*}
 f(y^N)+g(y^N) - f(x^\star)-g(x^\star) \leq \frac{1 }{4\gamma((N-1)\lambda + 1)}\|w^1 - w^*\|^2. 
\end{align*}
\end{conjecture}

It is worth noting that the convergence rate given in this conjecture for $\lambda = 1$ is similar to that established in \cite{zamani2024exact2} for the ADMM.

%%%%%%%%%%%%%%%%%%%%%%%%%%%%%%%%%%%%%%%
\subsection{Silver Stepsize Schedule}
In this subsection, we study the convergence rate of the DRS algorithm when using the silver stepsize schedule as a relaxation parameter. 
Let $\rho=1+\sqrt{2}$ denote the silver ratio. The silver stepsize schedule $\pi_k\in\mathbb{R}^{2^k-1}$ for $k\geq 1$ is defined as
$$
\pi_{k+1}=[\pi_{k}, \ \ 1+\rho^{k-1}, \ \ \pi_{k}],
$$
where $\pi_1=\sqrt{2}$.
In recent seminal work, Altschuler and Parrillo \cite{altschuler2024acceleration, altschuler2025acceleration} established that the convergence rate of the gradient method for smooth convex problems can be  improved by employing the silver stepsize schedule. They consider the convex optimization problem
\begin{align}\label{P_P}
\min_{x \in \mathbb{R}^n} F(x),
\end{align}
where $F:\mathbb{R}^n \to \mathbb{R}$ is $L$-smooth. They demonstrate that, by using stepsizes  $h = \frac{1}{L} \pi_k$, the gradient descent
\begin{align}\tag{GD}\label{GD}
  x^{i+1} = x^i - h_i\nabla F(x^i)  
\end{align}
satisfies the bound
\begin{align}
F(x^N) - F(x^\star) \leq \frac{L \|x^0 - x^\star\|^2}{\sqrt{1 + 4 \rho^{2k} - 3}}\approx  \frac{L \|x^0 - x^\star\|^2}{2 N^{\log_2{(\rho)}}},
\end{align}
where $N=2^k-1$, $x^\star \in \arg\min F(x)$ and $x^0$ is the initial point; see \cite[Theorem 1]{altschuler2024acceleration}. In what follows, we give a tighter rate for this stepszies\footnote{The proof was prepared in the fall of 2023, when the second author was a researcher at UCLouvain. The argument was developed jointly by the second author and François Glineur.}. 
Before we get to the proof we need to present a lemma. Moreover, we assume that $F$ is $1$-smooth without loss of generality and $F^\star=F(x^\star)$. In the lemma, we employ the following interpolation inequality  \cite[Theorem 2.1.5]{nesterov2018lectures},
$$
F(y)\geq F(x)+\langle \nabla F(x), y-x\rangle+\tfrac{1}{2}\|\nabla F(y)-\nabla F(x)\|^2.
$$

\begin{lemma}\label{lemma}
Let  $F$ be $1$-smooth and $N=2^k-1$. Assume that $x^1, \dots, x^N$ is generated by the GD with the silver stepsize schedule, $h=\pi_k$, and the initial point $x^0$. Then
\begin{align}\label{M.I}
\nonumber & \left(2\rho^k-1\right)\left(F^\star-\f{N}\right)
+\tfrac{1}{2}\left\|x^0-x^\star\right\|^2
-\sum_{i=0}^{N-1}h_i\left(  F^\star-\f{i}-\langle \g{i}, x^\star-x^i\rangle-\tfrac{1}{2}\left\| \g{i} \right\|^2 \right)\\
& -\rho^k\left(  F^\star-\f{N}-\langle \g{N}, x^\star-x^N\rangle-\tfrac{1}{2}\left\| \g{N}\right\|^2 \right)
-\tfrac{1}{2}\left\| x^{N}-\rho^k \g{N}-x^\star\right\|^2 \geq 0.
\end{align} 
\end{lemma}
\begin{proof}{Proof.}
 We prove by induction on $k$. Let $k=1$. By doing some algebra, one can show that 
\begin{align*}
& \rho\left( \f{0}-\f{1}-\langle \g{1}, x^0-x^1\rangle-\tfrac{1}{2}\left\|\g{1}-\g{0}\right\|^2 \right)+\\
&\left( \f{1}-\f{0}-\langle \g{0}, x^1-x^0\rangle-\tfrac{1}{2}\left\|\g{1}-\g{0}\right\|^2 \right)=\\
& \left(2\rho-1\right)\left(F^\star-\f{1}\right)
+\tfrac{1}{2}\left\|x^0-x^\star\right\|
-h_1\left(  F^\star-\f{0}-\tfrac{1}{h_1}\langle x^0-x^1, x^\star-x^0\rangle-\tfrac{1}{2h_0^2}\left\| x^0-x^1\right\|^2 \right)\\
& -\rho\left(  F^\star-\f{1}-\langle \g{1}, x^\star-x^1\rangle-\tfrac{1}{2}\left\| \g{1}\right\|^2 \right)
-\tfrac{1}{2}\left\| x^1-\rho \g{1}-x^\star\right\|^2\geq 0,
\end{align*}
where the equality follows from $\g{0}=\tfrac{1}{h_0}(x^0-x^1)$ and the inequality follows from the interpolation inequality. Assume that \eqref{M.I} holds for $k$. As $h_i=h_{N+1+i}$ for $i\in\{0, \dots, N-1\}$, by setting $x^{N+1}$ as the initial point in the GD, we get 
\begin{align}\label{M.II}
\nonumber & \left(2\rho^k-1\right)\left(F^\star-\f{2N+1}\right)
+\tfrac{1}{2}\left\|x^{N+1}-x^\star\right\|^2
-\sum_{i=N}^{2N}h_i\left(  F^\star-\f{i}-\langle \g{i}, x^\star-x^i\rangle-\tfrac{1}{2}\left\| \g{i}\right\|^2 \right)
\\
\nonumber & -\rho^k\left(  F^\star-\f{2N+1}-\langle \g{2N+1}, x^\star-x^{2N+1}\rangle-\tfrac{1}{2}\left\| \g{2N+1}\right\|^2 \right)
-\tfrac{1}{2}\left\|x^{2N+1}-\rho^k \g{2N+1}-x^\star\right\|^2
\\
 & +h_{N}\left(  F^\star-\f{N}-\langle \g{N}, x^\star-x^{N}\rangle-\tfrac{1}{2}\left\| \g{N}\right\|^2 \right)\geq 0.
\end{align}
In inequality \eqref{M.II}, we add and subtract the interpolation inequality between $x^{N}$ and $x^\star$ to facilitate subsequent computations. By the interpolation inequalities, we have 
\begin{align}\label{M.III}
\nonumber &\rho\sum_{i=N}^{2N}h_i\left(  \f{N}-\f{i}-\langle \g{i}, x^{N}-x^i\rangle-\tfrac{1}{2}\left\| \g{i}-\g{N}\right\|^2 \right)
\\ 
&+\rho\sum_{i=N}^{2N}h_i\left( \f{2N+1}-\f{i}-\langle \g{i},x^{2N+1}-x^i\rangle-\tfrac{1}{2}\left\| \g{i}-\g{2N+1}\right\|^2 \right)
\\ 
\nonumber & +\rho\left(  \f{N}-\f{2N+1}-\langle \g{2N+1}, x^{N}-x^{2N+1}\rangle-\tfrac{1}{2}\left\| \g{2N+1}-\g{N}\right\|^2 \right)
\\
\nonumber & -\rho\left( \f{2N+1}-\f{N}-\langle \g{N},x^{2N+1}-x^{N}\rangle-\tfrac{1}{2}\left\| \g{N}-\g{2N+1}\right\|^2 \right)\\
=
\nonumber &\rho\sum_{i=N}^{2N}h_i\left(  \f{N}-\f{i}-\langle \g{i}, x^{N}-x^i\rangle-\tfrac{1}{2}\left\| \g{i}-\g{N}\right\|^2 \right)
\\ 
\nonumber &+\rho\sum_{i=N+1}^{2N}h_i\left( \f{2N+1}-\f{i}-\langle \g{i},x^{2N+1}-x^i\rangle-\tfrac{1}{2}\left\| \g{i}-\g{2N+1}\right\|^2 \right)
\\ 
\nonumber & +\rho\left(  \f{N}-\f{2N+1}-\langle \g{2N+1}, x^{N}-x^{2N+1}\rangle-\tfrac{1}{2}\left\| \g{2N+1}-\g{N}\right\|^2 \right)
\\
\nonumber & +\rho^k\left( \f{2N+1}-\f{N}-\langle \g{N},x^{2N+1}-x^{N}\rangle-\tfrac{1}{2}\left\| \g{N}-\g{2N+1}\right\|^2 \right)
\geq 0,
\end{align}
Note that we use $h_{N}=\rho^{k-1}+1$. By multiplying \eqref{M.II} by $2\rho+1$ and then summing it with inequalities \eqref{M.I} and  \eqref{M.III}, we obtain
\begin{align*}
& (2\rho+1)\left(2\rho^{k}-1\right)\left(F^\star-\f{2N+1}\right)+
\tfrac{1}{2}\left\|x^0-x^\star\right\|^2
-\sum_{i=0}^{2N}h_i\left(  F^\star-\f{i}-\langle \g{i}, x^\star-x^i\rangle-\tfrac{1}{2}\left\| \g{i} \right\|^2 \right)
\\
& -\rho^{k+1}\left(  F^\star-\f{2N+1}-\langle \g{2N+1}, x^\star-x^{2N+1}\rangle-\tfrac{1}{2}\left\| \g{2N+1}\right\|^2 \right)
-\tfrac{1}{2}\left\| x^{N}-\rho^k \g{N}-x^\star\right\|^2
\\
&+\left(2\rho^k-1\right)\left(F^\star-\f{N}\right)
+(h_{N+1}(2\rho+1)-\rho^k)\left(  F^\star-\f{N}-\langle \g{N}, x^\star-x^{N}\rangle-\tfrac{1}{2}\left\| \g{N}\right\|^2 \right)
\\
& +(\rho+\tfrac{1}{2})\left\|x^{N}-h_{N}\nabla f(x^{N})-x^\star\right\|^2
-(\rho+\tfrac{1}{2})\left\|x^{2N+1}-\rho^k \g{2N+1}-x^\star\right\|^2
\\
& -(\rho+1)\rho^k\left(  F^\star-\f{2N+1}-\langle \g{2N+1}, x^\star-x^{2N+1}\rangle-\tfrac{1}{2}\left\| \g{2N+1}\right\|^2 \right)
\\
&+\rho\sum_{i=N}^{2N}h_i\left(  \f{N}-F^\star-\langle \g{i}, x^{N}-\g{N}-x^\star\rangle-\tfrac{1}{2}\left\| \g{N}\right\|^2 \right)
\\ 
&+\rho\sum_{i=N}^{2N}h_i\left( \f{2N+1}-F^\star-\langle \g{i},x^{2N+1}-\g{2N+1}-x^\star\rangle-\tfrac{1}{2}\left\|\g{2N+1}\right\|^2 \right)
\\ 
& +\rho\left(  \f{N}-\f{2N+1}-\langle \g{2N+1}, x^{N}-x^{2N+1}\rangle-\tfrac{1}{2}\left\| \g{2N+1}-\g{N}\right\|^2 \right)
\\
& -\rho\left( \f{2N+1}-\f{N}-\langle \g{N},x^{2N+1}-x^{N}\rangle-\tfrac{1}{2}\left\| \g{N}-\g{2N+1}\right\|^2 \right)
\geq 0.
\end{align*}
By using the following identities,
\begin{align*}
&  \sum_{i=N}^{2N} h_i=\rho^k+\rho^{k-1}, & \sum_{i=N}^{2N}h_i\g{i}=x^{N}-x^{2N+1},& 
& \rho^2-2\rho-1=0, 
\end{align*}
 we get
 \begin{align*}
\nonumber & \left(2\rho^{k+1}-1\right)\left(F^\star-\f{2N+1}\right)-
\sum_{i=0}^{2N}h_i\left(  F^\star-\f{i}-\langle \g{i}, x^\star-x^i\rangle-\tfrac{1}{2}\left\| \g{i} \right\|^2 \right)\\
& -\rho^{k+1}\left(  F^\star-\f{2N+1}-\langle \g{2N+1}, x^\star-x^{2N+1}\rangle-\tfrac{1}{2}\left\| \g{2N+1}\right\|^2 \right)+
\tfrac{1}{2}\left\|x^0-x^\star\right\|^2-\\
&\tfrac{1}{2}\left\| x^{2N+1}-\rho^{k+1} \g{2N+1}-x^\star\right\|^2 \geq 0,
\end{align*}
which completes the proof. 
\Halmos
\end{proof}
%%%%%%%%%%%%%%%%%%%%%%%%%%%%%%%%%%%%%%%%%%%%%%%%%%%%%%%%%%%%%%
\begin{theorem}
Let  $F$ be $1$-smooth and $N=2^k-1$. If $x^1, \dots, x^{N}$ is generated by the GD with the silver stepsize schedule and the initial point $x^0$, then
\begin{align}
\f{N}-F^\star\leq \tfrac{1}{4\rho^k-2}
\left\|x^0-x^\star\right\|^2.
\end{align} 
\end{theorem}
\begin{proof}{Proof.}
The theorem is derived from Lemma \ref{lemma}.
\Halmos
\end{proof}

One can show the rate given in the theorem is tight by selecting a proper Huber function. After the publication of \cite{altschuler2025acceleration}, several interesting developments have appeared. 
Grimmer et al.\cite{grimmer2025accelerated} improved the convergence rate by slightly modifying the silver stepsize schedule and also investigated the schedule for other performance measures. 
Bok and Altschuler \cite{bok2025optimized} studied the schedule in the context of the proximal gradient method. Wang et al.\cite{wang2024relaxed} analyzed it for the proximal point algorithm. 
We refer the interested reader to \cite{bok2025optimized} for a review of related work. The following conjecture concerns the convergence rate of the DRS algorithm when the silver stepsize schedule is employed to tune the relaxation parameter.

\begin{conjecture} \label{Conjecture_2}
Let Assumption \ref{assumption2} hold and let $\lambda=(\pi_k,  1)$. If the sequence $\{(x^k, y^k, w^k)\}$ is generated by Algorithm \ref{DRS_O} with stepsize $\gamma\in (0, \tfrac{2\sqrt{2}-1}{L})$, relaxation parameter $\lambda_i$ at iterate $i$ and  the initial point $w^1$, then 
\begin{align*}
 f(y^N)+g(y^N) - f(x^\star)-g(x^\star) \leq \frac{1}{4\gamma\rho^k}\|w^1 - w^*\|^2,
\end{align*}
where $N=2^k$.
\end{conjecture}

Note that since $\lambda_N$ does not influence $y^N$ in the DRS algorithm, it may be assigned arbitrarily as any positive scalar.

%%%%%%%%%%%%%%%%%%%%%%%%%%%%%%%%%%%%%%%%%%
\subsection{Accelerated DRS algorithm}
This subsection is devoted to an accelerated DRS algorithm. Several accelerated methods for problem \eqref{P_O} have been proposed in the literature; see, for example, \cite{goldstein2014fast, patrinos2014douglas}. However, their analyses establish an accelerated convergence rate of $O(\tfrac{1}{N^2})$ only in the quadratic setting. To address problem \eqref{P_O} in greater generality, we introduce Algorithm \ref{ADRS_O}.

\begin{algorithm}
\caption{Accelerated DRS algorithm}
\begin{algorithmic}
\State \textbf{Parameters:} number of iterations $N$,  positive stepsize $\gamma>0$ and relaxation factor $\lambda\in (0, 2)$.
\State \textbf{Inputs:} closed proper convex functions $f$ and $g$, initial iterate $w^1 \in \mathbb{R}^n$.
\State Set $u^1=w^1$ and for $k=1, 2, \ldots, N$ perform the following steps:\\
\begin{enumerate}[label=\roman*)]
\item 
$x^k = \operatorname{prox}_{\gamma f}(w^k)$.
\item
$y^k = \operatorname{prox}_{\gamma g}(2x^k - w^k)$.
\item
$u^{k+1} = w^k + \lambda (y^k - x^k)$.
\item
$w^{k+1} = u^{k+1} + \tfrac{k}{k+3} (u^{k+1} - u^k)$.
\end{enumerate}
\end{algorithmic}
\label{ADRS_O}
\end{algorithm}

Algorithm \ref{ADRS_O} closely resembles the method proposed in \cite{patrinos2014douglas}. The key distinction lies in the choice of the momentum coefficient: their algorithm, at iterate $k\geq 3$, employs \(\tfrac{k-2}{k+1}\) while Algorithm~\ref{ADRS_O} uses \(\tfrac{k}{k+3}\). 
It is also worth noting that \cite{patrinos2014douglas} sets the coefficient to zero for the first two iterates.
Based on our numerical experiments, we propose the following conjecture regarding the convergence rate of Algorithm \ref{ADRS_O}.

\begin{conjecture} \label{Conjecture_4}
Let Assumption \ref{assumption2} hold. If the sequence $\{(x^k, y^k, w^k, u^k)\}$ is generated by Algorithm \ref{ADRS_O} with stepsize $\gamma\in (0, \tfrac{1}{L}]$, relaxation parameter $\lambda\in (0, 1]$ and the initial point $w^1$, then 
\begin{align*}
 f(y^N)+g(y^N) - f(x^\star)-g(x^\star) \leq \tfrac{2}{\gamma((N^2+7N-8)\lambda + 8)} \|w^1 - w^*\|^2. 
\end{align*}
\end{conjecture}
%%%%%%%%%%%%%%%%%%%%%%%%%%%%%%%%%%%%%%%%%
\bibliographystyle{informs2014} % outcomment this and next line in Case 1
\bibliography{sample} % if more than one, comma separated

% CASE 2: BiBTeX used to generate mypaper.bbl (to be further fine tuned)
%\input{mypaper.bbl} % outcomment this line in Case 2

%If you don't use BiBTex, you can manually itemize references as shown below.

%\bibliographystyle{nonumber}

%\begin{thebibliography}{3}
%\providecommand{\natexlab}[1]{#1}
%\providecommand{\url}[1]{\texttt{#1}}
%\providecommand{\urlprefix}{URL }

%\bibitem[{Smith(2005)}]{smith2005}
%Smith J (2005) Optimal resource allocation in humanitarian logistics.
%  \emph{Journal of Operations Research} 30(2):123--135.
  
%\bibitem[{Jones(2010)}]{jones2010}
%Jones S (2010) Stochastic programming models for humanitarian logistics.
%  \emph{INFORMS Mathematics of Operations Research} 35(4):567--580.

%\bibitem[{Brown(2015)}]{brown2015}
%Brown D (2015) \emph{Introduction to Stochastic Programming} (Springer).

%\end{thebibliography}

%%%%%%%%%%%%%%%%%
\end{document}